\documentclass[11pt,a4paper]{article}

\usepackage{geometry}
\usepackage[T1]{fontenc}
\usepackage[utf8]{inputenc}
\usepackage{authblk}
\usepackage{amsmath}
\usepackage{amssymb}
\usepackage{amsfonts}
\usepackage{amsthm}
\usepackage{enumerate}
\usepackage{paralist}
\usepackage{amsfonts}
\usepackage{tensor} 
\usepackage{lmodern}
\usepackage{upgreek}
\usepackage{mathrsfs}
\usepackage{nicefrac}
\usepackage{verbatim}
\usepackage{amsmath}
\usepackage[all, cmtip]{xy}\usepackage{amssymb}
\usepackage{amsthm}
\usepackage{extpfeil}
\usepackage{bbm}
\usepackage{paralist}
\usepackage[colorlinks,linkcolor=blue,citecolor=gray,pagebackref=true,urlcolor=blue,pdftex]{hyperref}

\usepackage{titlesec}
\titleformat*{\section}{\Large\sffamily}

\usepackage{graphicx}

\usepackage[all]{xy}
\usepackage{tikz}
\usepackage{relsize}
\usepackage{comment}
\usepackage{mathtools}

\theoremstyle{plain}
\newtheorem{thm}{Theorem}

\newtheorem*{thm*}{Theorem}

\newtheorem{prop}[thm]{Proposition}
\newtheorem{lem}[thm]{Lemma}
\newtheorem{cor}[thm]{Corollary}

\newtheorem*{prb*}{Problem}

\theoremstyle{definition}
\newtheorem{rmk}[thm]{Remark}
\newtheorem{df}[thm]{Definition}

\newtheorem{ex}[thm]{Example}

\newcommand\note[1]{\textbf{\color{red}#1}}
\newcommand\Defn[1]{\emph{\color{blue}#1}}
\newcommand{\cm}[1]{}

\newcommand\wt[1]{\widetilde{#1}}

\renewcommand{\k}{\mathbbm{k}}

\newcommand{\lir}{{\langle i\rangle}\,}
\newcommand{\dps}{\mathrm{depth}\,}

\newcommand\BW{\mathrm{BW}}

\newcommand\U{\mathrm{U}}

\newcommand\h{\mathrm{h}}

\makeatletter
\def\dual#1{\expandafter\dual@aux#1\@nil}
\def\dual@aux#1/#2\@nil{\begin{tabular}{@{}c@{}}#1\\#2\end{tabular}}
\makeatother

\begin{document}

\title{Dimension filtration, sequential Cohen--Macaulayness and a new polynomial invariant of graded algebras}

\author{Afshin Goodarzi%
  \thanks{Email address: \texttt{afshingo@kth.se}}}
\affil{ Royal Institute of Technology, Department of Mathematics, S-100 44, Stockholm, Sweden.}
\maketitle


\begin{abstract}
\noindent Let $\k$ be a field and let $A$ be a standard $\mathbb{N}$-graded $\k$-algebra. Using numerical information of some invariants in the primary decomposition of $0$ in $A$, namely the so called dimension filtration, we associate a bivariate polynomial $\BW(A;t,w)$, that we call the Bj\"{o}rner--Wachs polynomial, to $A$.\\
 It is shown that the Bj\"{o}rner--Wachs polynomial is an algebraic counterpart of the combinatorially defined $h$-triangle of finite simplicial complexes introduced by Bj\"{o}rner \& Wachs. We provide a characterisation of sequentially Cohen--Macaulay algebras in terms of the effect of the reverse lexicographic generic initial ideal on the Bj\"{o}rner--Wachs polynomial. More precisely, we show that a graded algebra is sequentially Cohen--Macaulay if and only if it has a stable Bj\"{o}rner--Wachs polynomial under passing to the reverse lexicographic generic initial ideal. We conclude by discussing connections with the Hilbert series of local cohomology modules.
\end{abstract}


\section{Introduction}


 Associated to every finite simplicial complex there is a standard monomial algebra, the so called face ring of the complex. In order to verify the upper bound conjecture, Stanley~\cite{Stanley75} studied the face numbers of a triangulated sphere via the numerical properties of this algebra. Stanley's proof has two major ingredients. Namely, that the Hilbert series of a face ring can be expressed in terms of the combinatorially defined $h$-numbers of the complex, and that the face ring of a triangulated sphere is Cohen--Macaulay. One may have a look at Stanley's recent article~\cite{StanleyUBCstory} for the full account of the story of ``\textit{how the upper bond conjecture was proved}''.
 
 The Cohen--Macaulay property of the face ring is implied by the combinatorial property of the complex being shellable. Motivated by examples coming from subspace arrangements, Bj\"{o}rner \& Wachs~\cite{Bj-WI} generalised shellability by introducing the concept of nonpure shellable complexes. Stanley ~\cite{StanleyGreen} then introduced sequential Cohen--Macaulayness in order to have an algebraic counterpart for this new concept. At almost the same time Schenzel~\cite{Schenzel99}, independently, came up with the same notion but from a totally different perspective. 
 
 Bj\"{o}rner \& Wachs~\cite{Bj-WI} also defined doubly indexed $h$-numbers of a simplicial complex as a finer invariant than the usual $h$-numbers, in the sense that one can obtain the latter from the former (see Section~\ref{faceRing} below, for precise definitions and more details). The array of doubly indexed $h$-numbers of a complex is called the $h$-triangle. For a sequentially Cohen--Macaulay complex some interesting topological and algebraic invariants, such as topological Betti numbers of the complex and graded Betti numbers of the face ring of its Alexander dual, are encoded in the $h$-triangle. It should be noted that this later connection has been recently used by Adiprasito, Bj\"{o}rner \& Goodarzi~\cite{KAA} to characterise the possible Betti tables of componentwise linear ideals. 
 
 As the Hilbert series is the algebraic counterpart of the $h$-numbers, one would expect to have an algebraic counterpart also for the doubly indexed $h$-numbers. The objective of this paper is to fill this gap by providing an algebraic counterpart for the $h$-triangle. By making use of Schenzel's dimension filtration~\cite{Schenzel99}, to every standard $\mathbb{N}$-graded $\k$-algebra we associate a bivariate polynomial; the Bj\"{o}rner--Wachs polynomial. This polynomial specialises to the $h$-triangle in the case of face rings of simplicial complexes. The Bj\"{o}rner--Wachs polynomial of a sequentially Cohen--Macaulay algebra contains much interesting information of the algebra, such as extremal Betti numbers.

 The paper is organised as follows. First in Section~\ref{Pre}, we recall some preliminaries and define the Bj\"{o}rner--Wachs polynomial. The face rings are the subject of study in Section~\ref{faceRing}. We show in Theorem~\ref{BWDelta} that the combinatorially defined doubly indexed $h$-numbers of a simplicial complex are, precisely, the coefficients of the Bj\"{o}rner--Wachs polynomial of its face ring. We present some basic results about Borel-fixed ideals in Section~\ref{SSI}. The material in this section will be of use in the next sections. Section~\ref{charSCM} is devoted to a characterisation of sequentially Cohen--Macaulay algebras in terms of the Bj\"{o}rner--Wachs polynomial, namely, we prove in Theorem~\ref{sqcm} that sequentially Cohen--Macaulay algebras are exactly those that have a stable Bj\"{o}rner--Wachs polynomial under passing to the reverse lexicographic generic initial ideal. This result provides a (generalised) symmetric version of the main result in~\cite{Duval}. Beside this, in Proposition~\ref{char}, we give a few conditions, each of them being equivalent to sequential Cohen--Macaulayness. We will discuss some connections to the numerical data of the local cohomology modules in Section~\ref{Local} in case of sequentially Cohen--Macaulay algebras. Finally, in Section~\ref{Remarks} some remarks on the Alexander dual of sequentially Cohen--Macaulay simplicial complexes will be discussed.


\section{Preliminaries}\label{Pre}

We assume some familiarity with basic notions in commutative algebra and the theory of simplicial complexes. The reader is referred to the books by Eisenbud~\cite{Eisenbud} and by Stanley~\cite{StanleyGreen}, for undefined terminologies. 


\subsection{Dimension Filtration and Unmixed Layers}

Let $R=\k[x_1,\ldots, x_n]$ be the polynomial ring over a field $\k$. Assume that $R$ is equipped with the standard grading, i.e., $\deg x_i=1$ for all $1\leq i\leq n$. Let $I$ be a homogeneous ideal in $R$ and $A=R/I$ be the associated $\mathbb{N}$-graded algebra. Also, let $0=\bigcap_{j=0}^s \mathfrak{q}_j$ be a reduced primary decomposition of $I$. For $1\leq j\leq s$ Denote by $\mathfrak{p}_j$ the radical $\sqrt{\mathfrak{q}_j}$ of $\mathfrak{q}_j$. Let us denote by $I^{\langle i\rangle}$ the ideal
\begin{equation*}\label{pdecom}
I^{\langle i\rangle}=\bigcap_{\mathrm{dim}(R/\mathfrak{p}_j)>i} \mathfrak{q}_j,
\end{equation*}
 for all $0\leq i\leq d=\dim A$. In particular, we have $I=I^{\langle 0\rangle}$ and $I^{\langle d\rangle}=R$ and we have a filtration 
 
\begin{equation}\label{1filtr}
 I=I^{\langle 0\rangle}\subseteq I^{\langle 1\rangle}\subseteq\ldots\subseteq I^{\langle d\rangle}=R
\end{equation}
 that we call the \Defn{dimension filtration} of the pair $(I,R)$. After modding $I$ out of the components of the filtration~\eqref{1filtr}, we obtain 
\begin{equation}\label{2filtr}
 0\subseteq I^{\langle 1\rangle}/I\subseteq\ldots\subseteq I^{\langle d-1\rangle}/I\subseteq A,
\end{equation}
the \Defn{dimension filtration} of $A$. Notice that, if $I$ is not a radical ideal, then the primary decomposition is not necessarily unique. However, the following result (see~\cite[Proposition 2.2]{Schenzel99} and \cite[pp 100-103]{Eisenbud}) shows that the dimension filtration is independent from the choice of primary components and thus is unique. 
 
 \begin{lem}\label{ess} For a submodule $N$ of $A=R/I$ the following are equivalent
 \begin{enumerate}[\rm (a)]
 \item $N$ is the largest submodule of $A$ such that $\dim N\leq i$,
 \item $N=\mathrm{H}^0_{\mathfrak{a}_i}(A)$, where $\mathfrak{a}_i=\prod \{\mathfrak{p}\mid \mathfrak{p}\in\mathrm{Ass} I\ \& \ \dim R/\mathfrak{p}\leq i\}$,
 \item $N\cong A^\lir:=I^\lir/I$.
 \end{enumerate}
 In particular, filter ideals $I^\lir$ are independent from the choice of primary components. 
 \end{lem}\qed
 
\begin{df}
Let $A=R/I$ be a standard $\mathbb{N}$-graded algebra. Then the $i$-th \Defn{unmixed layer} of $A$ is defined to be the $R$-module \(\mathit{U}_i(A)=I^{\langle i\rangle}/I^{\langle i-1\rangle}\)
for all $1\leq i\leq \dim A$.
\end{df}
\noindent The name we give to these modules is motivated by the following result.
\begin{lem}[{\cite[Corollary 2.3]{Schenzel99}}]The unmixed layers of $A$ are unmixed, that is, for any $i$ the associated primes of $\U_i(A)$ have the same height. Furthermore, $A$ is unmixed if and only if $\U_i(A)=0$ for all $i<d=\dim A$ and $\U_d(A)=A$.

\end{lem}

\subsection{Hilbert Series and Bj\"{o}rner--Wachs Polynomial}
\noindent 
Let $A=R/I$ be a standard $\mathbb{N}$-graded algebra of Krull dimension $d$. Then there exists a polynomial $\mathrm{h}(A;t)\in\mathbb{Z}[t]$, the \Defn{$h$-polynomial} of $A$ such that the Hilbert series of $A$ has the following rational expression:
\begin{equation*}
\mathrm{Hilb}(A;t)=\frac{\mathrm{h}(A;t)}{(1-t)^d }.
\end{equation*} Let us also denote by $h_j(A)$ the coefficient of $t^j$ in $h(A;t)$.\\We now define our main object of study in this paper. To do so, the following lemma is needed. 

\begin{lem}
The $i$-th unmixed layer $\U_i(A)$ of $A$ is either zero or has Krull dimension $i$.
\end{lem}
\begin{proof}
Consider the short exact sequence \[0\rightarrow A^{\langle i-1\rangle}\rightarrow A^{\langle i\rangle}\rightarrow\U_i(A)\rightarrow 0\] of graded $R$-modules. If $\dim A^{\langle i\rangle}<i$, then $A^{\langle i\rangle}=A^{\langle i-1\rangle}$ and hence $\U_i(A)=0$. Otherwise by applying the depth lemma~\cite[Proposition 1.2.9]{Bruns-Herzog}, we get \[\max\left(\dim A^{\langle i-1\rangle} , \dim\U_i\left(A\right)\right)= \dim A^{\langle i\rangle}=i.\]
\end{proof}

\begin{df}\label{BW} Let $A$ be a standard $\mathbb{N}$-graded $\k$-algebra of Krull dimension $d$. The bivariate polynomial 
\begin{equation}
\mathrm{BW}(A;t,w)=\sum_{i=0}^d \mathrm{h}(\U_i(A);t) w^i
\end{equation}
is called the
\Defn{Bj\"{o}rner--Wachs polynomial} of $A$. 
\end{df}

Note that if $A$ is an unmixed algebra, then $\U_d(A)=A$ and $\U_i(A)=0$ for $i\neq d$. Thus in this situation knowing the Hilbert series and the Bj\"{o}rner--Wachs polynomial are equivalent, more precisely one has $\BW(A;t,w)=\mathrm{h}(A;t)w^d$. This, of course, is not true in general. While we have

\begin{prop}\label{specialise}  The Bj\"{o}rner--Wachs polynomial specialises to the Hilbert series by putting $w:=1/(1-t)$, that is,
\[\mathrm{Hilb}(A;t)=\BW(A;t,1/(1-t)).\]
\end{prop}
\begin{proof}
From the short exact sequence \[0\rightarrow A^{\langle i-1\rangle}\rightarrow A^{\langle i\rangle}\rightarrow\U_i(A)\rightarrow 0\] of graded $R$-modules, we get 
\[\mathrm{Hilb}(A^{\langle i\rangle};t)=\mathrm{Hilb}(A^{\langle i-1\rangle};t)+\mathrm{Hilb}(\U_i(A);t).\] Now summing over all $i$ we have \[\mathrm{Hilb}(A;t)=\sum_i \mathrm{Hilb}(\U_i(A);t)=\sum_i\frac{ \mathrm{h}(\U_i(A);t)}{(1-t)^i}.\]
\end{proof}

It is well known that the Hilbert series of $A$ can be computed from its Betti table (i.e. the table of graded Betti numbers). However, unlike the Hilbert series, the Bj\"{o}rner--Wachs polynomial can not be obtained from the graded Betti numbers. To show this, one may take $A=R/I$ to be componentwise linear but not sequentially Cohen--Macaulay. Then the Betti table of $A$ is stable under passing to the generic initial ideal with respect to reverse lexicographic order, but the Bj\"{o}rner--Wachs polynomial is not stable. See Theorem~\ref{sqcm} below and \cite[Theorem 8.2.22]{Herzog-Hibi} for more information. The following concrete example is computed using Macaulay2~\cite{M2}. 

\begin{ex}

Let $R=\mathbb{Q}[x_1,x_2,x_3,x_4,x_5,x_6]$ and let \[I=\langle x_1x_2x_3,x_1x_4,x_2x_5,x_3x_6,x_4x_5,x_4x_6,x_5x_6 \rangle\] be an ideal of $R$. Then the generic initial ideal of $I$ with respect to the reverse lexicographic order is \[\mathrm{gin}(I)=\langle x_1^2,x_1x_2,x_2^2,x_1x_3,x_2x_3,x_3^2,x_1x_4^2 \rangle.\] The Betti tables of $R/I$ and $R/\mathrm{gin}(I)$ 
\[
\begin{tabular}{c|ccccc}
&1 & 7&11&6&1 \\
\hline
0:&1 & $\cdot$ & $\cdot$ & $\cdot$ & $\cdot$\\
1:& $\cdot$ &6 & 8 & 3 & $\cdot$\\
2:& $\cdot$ & 1 & 3 & 3 & 1
\end{tabular}
\]
coincide; $I$ is componentwise linear. On the other hand, one has $\BW(R/I;t,w)=w^3+tw^3-t^3w^3$ while $\BW(R/\mathrm{gin}(I);t,w)=tw^2+w^3+t^2w^2+2tw^3$.

\end{ex}


\section{The Face Rings of Simplicial Complexes}\label{faceRing}
In this section we study the Bj\"{o}rner--Wachs polynomial of the face rings of simplicial complexes. \\ Let us start by briefly recalling some notions. A \Defn{simplicial complex} $\Delta$ on the vertex set $[n]:=\{1,2,\ldots,n\}$ is a subset of the power set $2^{[n]}$ that is closed under passing to subsets, that is, if $F\in\Delta$ and $G\subseteq F$ then $G\in\Delta$. The elements of $\Delta$ are called \Defn{faces}. The inclusion-wise maximal faces are called \Defn{facets}. The set of facets of $\Delta$ is denoted by $\mathcal{F}(\Delta)$. Clearly, a simplicial complex is uniquely determined from its set of facets. The \Defn{dimension} $\dim F$ of a face $F$ in $\Delta$ is defined to be one less than its cardinality and the \Defn{dimension} of $\Delta$ itself is equal to the maximum of dimension of its faces. 

Let $\Delta$ be a $(d-1)$-dimensional simplicial complex on the vertex set $[n]$. The \Defn{Stanley--Reisner} ideal $I_\Delta$ of $\Delta$ is the ideal of $R$ generated by all monomials $x_G:=\Pi_{i\in G} x_i$, where $G\notin\Delta$. The quotient ring $\k[\Delta]=R/I_\Delta$ is called the \Defn{face ring} of $\Delta$. It can be shown that $\k[\Delta]$ has Krull dimension equal to $d=\dim\Delta+1$.

 Let us denote by $f_i(\Delta)$ the number of $i$-dimensional faces of $\Delta$. Then the Hilbert series of $\k[\Delta]$ can be computed from these combinatorial invariants~\cite[Proposition 6.2.1]{Herzog-Hibi}. In particular, one has 
\begin{equation*}
\h(\k[\Delta];t)=\sum f_{i-1}(\Delta) (1-t)^{d-i}t^i.
\end{equation*}
For a face $\sigma\in \Delta$ let the \Defn{degree} of $\sigma$ be defined as the biggest cardinality of the faces containing $\sigma$. Bj\"{o}rner \& Wachs~\cite{Bj-WI} defined the doubly indexed $f$-number $f_{i,j}(\Delta)$ to be the number of faces of $\Delta$ of degree $i$ and cardinality $j$. They also defined the \Defn{$h$-triangle} $\mathfrak{h}(\Delta)$ of $\Delta$ to be the triangular integer array $\mathfrak{h}(\Delta)=(h_{i,j}(\Delta))_{0\leq j\leq i\leq d}$, where
\begin{equation*}
 h_{i,j}(\Delta)=\sum_k(-1)^{j-k}{i-k\choose j-k}f_{i,k}(\Delta).
\end{equation*}

\begin{lem}\label{triangle} If $\Delta$ is a simplicial complex, then the doubly indexed $f$- and $h$- numbers of $\Delta$ satisfy the following equation:
\[\sum_i\sum_j h_{i,j}(\Delta)w^it^j=\sum_i\sum_j f_{i,j}(\Delta)w^it^j(1-t)^{i-j}.\]
\end{lem}

\begin{proof}
The proof is straightforward and we leave it to the reader. 
\end{proof}

The reduced primary decomposition of Stanley--Reisner ideals has a simple description. For a subset $G$ of $[n]$ denote by $\mathfrak{p}_G$ the monomial prime ideal generated by all $x_i$ such that $i\notin G$. It can be shown (see~\cite[Lemma 1.5.4]{Herzog-Hibi}, for instance) that

\[I_\Delta=\bigcap_{F\in\mathcal{F}(\Delta)}\mathfrak{p}_F\] 
is the unique reduced primary decomposition of $I_\Delta$. 

The dimension filtration of $\k[\Delta]$ can be described also by combinatorial means. For a simplicial complex $\Delta$ let us denote by $\Delta^{\langle i\rangle}$ the subcomplex of $\Delta$ generated by all facets of dimension $\geq i$, that is, $\mathcal{F}(\Delta^{\langle i\rangle})=\left\{F\in\mathcal{F}(\Delta)\mid \dim F\geq i\right\}$.
\begin{lem} Let $\Delta$ be a simplicial complex. Then the $i$-th unmixed layer $\left(\k[\Delta]\right)^{\langle i\rangle}$ of its face ring is isomorphic to $I_{\Delta^{\langle i\rangle}}/I_\Delta$.
\end{lem}
\begin{proof} The ideal $\left(I_\Delta\right)^\lir$ is the intersection of those primes $\mathfrak{p}_F$ such that $\dim(R/\mathfrak{p}_F)>i$. However, $\dim\left(R/\mathfrak{p}_F\right)$ is equal to the cardinality of $F$. Hence, $\dim(R/\mathfrak{p}_F)>i$ if and only if $F\in\mathcal{F}\left(\Delta^\lir\right)$ and one has $\left(I_\Delta\right)^\lir=I_{\Delta^{\langle i\rangle}}$.

\end{proof}
\begin{thm}\label{BWDelta} Let $\Delta$ be a simplicial complex. Then the Bj\"{o}rner--Wachs polynomial of $\k[\Delta]$ can be computed from the $f$-triangle of $\Delta$ via $$\BW(\k[\Delta];t,w)=\sum f_{i,j}(\Delta)w^it^j(1-t)^{i-j}.$$ In particular, $h_{i,j}(\Delta)$ is the coefficient of $w^it^j$ in the Bj\"{o}rner--Wachs polynomial of $\k[\Delta]$. 

\end{thm}
\begin{proof}
First, note that the coefficient of $w^i$ in the Bj\"{o}rner--Wachs polynomial of $\k[\Delta]$ is 
\[\Big[\BW(\k[\Delta];t,w)\Big]_{w^i}=\h(\U_i(\k[\Delta]);t)=\h\left(I_{\Delta^{\langle i\rangle}}/I_{\Delta^{\langle i-1\rangle}};t\right),\]
since, we have $\U_i(\k[\Delta])\cong I_{\Delta^{\langle i\rangle}}/I_{\Delta^{\langle i-1\rangle}}$ as graded $R$-modules. 
Now from the exact sequence 
\[0\rightarrow  \k[\Delta^{\langle i\rangle}]\rightarrow  \k[\Delta^{\langle i-1\rangle}]\rightarrow  I_{\Delta^{\langle i\rangle}}/I_{\Delta^{\langle i-1\rangle}}\rightarrow 0\] of graded $R$-modules, it follows that 
\[\mathrm{Hilb}\left(I_{\Delta^{\langle i\rangle}}/I_{\Delta^{\langle i-1\rangle}};t\right)=\mathrm{Hilb}(\k[\Delta^{\langle i-1\rangle};t)-\mathrm{Hilb}(\k[\Delta^{\langle i-1\rangle}];t).\] Thus
\[\h(I_{\Delta^{\langle i\rangle}}/I_{\Delta^{\langle i-1\rangle}};t)=\frac{\h(\k[\Delta^{\langle i-1\rangle}];t)-\h(\k[\Delta^{\langle i\rangle}];t)}{(1-t)^{d-i}}.\]
We can now conclude that 
\[\Big[\BW(\k[\Delta];t,w)\Big]_{w^i}=\sum_j \left(f_{j-1}(\Delta^{\langle i-1\rangle})-f_{j-1}(\Delta^{\langle i\rangle})\right)(1-t)^{i-j}t^j.\]
However, $f_{j-1}(\Delta^{\langle i-1\rangle})-f_{j-1}(\Delta^{\langle i\rangle})$ is equal to the number of faces of $\Delta$ of cardinality $j$ and degree $i$. The result now follows by using Lemma~\ref{triangle}.

\end{proof}


\section{Generic Initial Ideal and Borel-fixed Ideals}\label{SSI}

In the sequel we discuss some basic properties of Borel-fixed ideals. These properties will be of use in the next section. 

We are only concerned with the reverse lexicographic order induced by $x_1<\ldots<x_n$ on the set $\mbox{Mon}(R)$ of all monomials in $R$. In particular, $\mathrm{in}(I)$ and $\mathrm{gin}(I)$ denote the initial and generic initial ideal of $I$ with respect to this total ordering, respectively. The reader may consult~\cite[Chapter 15]{Eisenbud} or ~\cite{Green} for the precise definition and properties of the generic initial ideal and Borel-fixed ideals.

\begin{lem}\label{dpsborel}
Let $J\subseteq R$ be a Borel-fixed ideal. Then $\dps\left(R/J\right)$ is equal to the smallest integer $i$ such that $J$ is a proper subset of $J^\lir$.
\end{lem}
\begin{proof}

If $\mathfrak{p}$ is the maximal associated prime of $J$, then the condition in the statement is equivalent to $\dim R/\mathfrak{p}=i$. Thus, it follows from~\cite[corollary 15.25]{Eisenbud} that $\mathfrak{p}=\langle x_1,\ldots,x_{n-i}\rangle$ and $x_{n-i+1},\ldots,x_n$ is a maximal $(R/J)$-regular sequence. In particular $\dps\left(R/J\right)=i$.

\end{proof}

\begin{lem}\label{lemdim}
Let $I\subseteq J$ be two homogeneous ideals in $R$. Then $\dim\left(\mathrm{gin}(J)/\mathrm{gin}(I)\right)\leq\dim\left(J/I\right)$.
\end{lem}
\begin{proof}
 After a generic change of coordinates, we may assume that $\mathrm{gin}(I)=\mathrm{in}(I)$ and $\mathrm{gin}(J)=\mathrm{in}(J)$. We have
\[
\dim\left(J/I\right)=\dim\left(R/I:J\right)=\dim\left(R/\mathrm{in}(I:J)\right).
\]
Thus, it suffices to show that $\mathrm{in}(I:J)\subseteq \mathrm{in}(I):\mathrm{in}(J)$. Let $f$ be a polynomial in $I:J$. Then for all $g\in J$, the polynomial $g\cdot f$ belongs to $I$. Hence, $\mathrm{in}(g\cdot f)=\mathrm{in}(g)\cdot\mathrm{in}(f)$ belongs to $\mathrm{in}(I)$ for all $g\in J$. Therefore, we have $\mathrm{in}(f)$ belongs to $\mathrm{in}(I):\mathrm{in}(J)$. 

\end{proof}

\begin{prop}\label{gindim}
Let $I$ be a homogeneous ideal in $R$. Then $\mathrm{gin}\left(I^\lir\right)\subseteq\mathrm{gin}\left(I\right)^\lir$ for all $0\leq i<\dim\left(R/I\right)$.
\end{prop}

\begin{proof} 
If $I^\lir=I$, then there is nothing to prove. Otherwise, $I^\lir/I$ has Krull dimension $i$. So, it follows from Lemma~\ref{lemdim} that 
\(
\dim\left(\mathrm{gin}(I^\lir)/\mathrm{gin}(I)\right)
\)
is less than or equal to $i$. It follows now from Lemma~\ref{ess} that 
\(
\left(\mathrm{gin}(I^\lir)/\mathrm{gin}(I)\right)\subseteq \left(\mathrm{gin}(I)^\lir/\mathrm{gin}(I)\right).
\) Therefore, $\mathrm{gin}\left(I^\lir\right)$ is a subset of $\mathrm{gin}\left(I\right)^\lir$ as desired.  

\end{proof}

For an ideal $J$ of $R$ and a polynomial $f\in R$, let us denote by $(J:f^{\infty})$ the \Defn{saturation} of $I$ with respect to $f$, that is, the ideal
\begin{equation*}
(J:f^{\infty})=\left\{g\in R\mid   f^{s}\cdot g\in I\text{ for some }s>0\right\}.
\end{equation*}
Next, we describe the dimension filtration of Borel-fixed ideals. 

\begin{lem}\label{saturation}
Let $J$ be a Borel-fixed ideal and assume that $R/J$ has Krull dimension $d$. Then one has
\(J^{\langle i+1\rangle}=(J^\lir:x_{n-i}^{\infty})\) for all $0\leq i\leq d-1$.
\end{lem}
\begin{proof}
There are two possibilities that we treat separately: either $J^{\langle i+1\rangle}=J^\lir$ or $J^\lir$ is a proper subset of $J^{\langle i+1\rangle}$.
\paragraph*{Case 1: ($J^{\langle i+1\rangle}=J^\lir$).} In this case the variable $x_{n-i}$ cannot appear in the minimal generators of $J^\lir$. Thus, one has $J^{\langle i+1\rangle}=J^\lir=(J^\lir:x_{n-i}^{\infty})$. 
\paragraph*{Case 2: ($J^{\langle i+1\rangle}\neq J^\lir$).} In this case, there exists an $(x_1,\ldots,x_{n-i})$-primary ideal $\mathfrak{q}$ such that $J^{\langle i+1\rangle}\cap \mathfrak{q}=J^\lir$. Since a power of $x_{n-i}$ should be among the minimal generators of $\mathfrak{q}$, we have that $(\mathfrak{q}:x_{n-i}^{\infty})=R$. Thus, one has $(J^\lir: x_{n-i}^{\infty})\subseteq (J^{\langle i+1\rangle}:x_{n-i}^{\infty})\cap (\mathfrak{q}:x_{n-i}^{\infty})=J^{\langle i+1\rangle}$.\\
On the other hand, we have 
$J^\lir=\left(J^{\langle i+1\rangle}\cap \mathfrak{q}\right)\supseteq \left(J^{\langle i+1\rangle}\cdot \mathfrak{q}\right)$. Hence, one has also
\[(J^\lir: x_{n-i}^{\infty})\supseteq (J^{\langle i+1\rangle}\cdot \mathfrak{q}: x_{n-i}^{\infty})\supseteq J^{\langle i+1\rangle},\]
as desired.
\end{proof}

\begin{cor}\label{zir}
Let $I\subseteq J$ be two Borel-fixed ideals in $R$. Then $I^\lir\subseteq J^\lir$ for all $0\leq i\leq \dim(R/J)$.
\end{cor}
\begin{proof}
The assertion follows easily from Lemma~\ref{saturation} by using induction on $i$.
\end{proof}




\section{Sequentially Cohen--Macaulay Algebras}\label{charSCM}

In this section, we characterise sequentially Cohen--Macaulay standard $\mathbb{N}$-graded algebras by means of the effect of the generic initial ideal on the Bj\"{o}rner--Wachs polynomial. More precisely, we show that $A$ is sequentially Cohen--Macaulay if and only if it has stable Bj\"{o}rner--Wachs polynomial under passing to generic initial ideal with respect to reverse lexicographic order. Recall that, if $I$ is a homogeneous ideal of $R$, then $A=R/I$ is \Defn{sequentially Cohen--Macaulay} if and only if for all $0\leq i\leq \dim(R/I)$, the $i$-th unmixed layer $\U_i(A)$ of $A$ is either zero or Cohen--Macaulay of dimension $i$. Note that if $J$ is Borel-fixed, then $A=R/J$ is sequentially Cohen--Macaulay. In particular, $R/\mathrm{gin}(I)$ is sequentially Cohen--Macaulay for all homogeneous ideals $I\subseteq R$.  

We start by giving a few conditions equivalent to the property of being sequentially Cohen--Macaulay. 

\begin{prop}\label{char}
Let $I\subseteq R$ be a homogeneous ideal and assume that $A=R/I$ has Krull dimension $d$. Then the following are equivalent:
\begin{enumerate}[\rm (a)]
\item\label{parta} $A$ is sequentially Cohen--Macaulay;
\item\label{partb} $\dps\left(R/I^\lir\right)\geq i+1$ for all $0\leq i<d$;
\item\label{partcc} $\mathrm{gin}\left(I^\lir\right)=\mathrm{gin}\left(I^\lir\right)^\lir$ for all $0\leq i<d$;
\item\label{partc} $\mathrm{gin}\left(I^\lir\right)=\mathrm{gin}\left(I\right)^\lir$ for all $0\leq i<d$;
\item\label{partd} $\mathrm{Hilb}(R/\mathrm{gin}(I^\lir);t)=\mathrm{Hilb}(R/\mathrm{gin}(I)^\lir;t)$ for all $0\leq i<d$;
\item\label{parte} $\mathrm{Hilb}(R/I^\lir;t)=\mathrm{Hilb}(R/\mathrm{gin}(I)^\lir;t)$ for all $0\leq i<d$.

\end{enumerate}
\end{prop}
\begin{proof}
First note that the implications \eqref{partd} $\iff$ \eqref{parte} and \eqref{partc} $\implies$ \eqref{partd} are clear. On the other hand, by Proposition~\ref{gindim} we know that $\mathrm{gin}\left(I^\lir\right)$ is always a subset of $\mathrm{gin}\left(I\right)^\lir$. Hence, the equality between the Hilbert series forces the ideals to be the same. So, the parts \eqref{partc}, \eqref{partd} and \eqref{parte} are all equivalent.\\
Now consider the short exact sequence 
\[
0\rightarrow R/I^{\langle i+1\rangle}\rightarrow R/I^{\langle i\rangle}\rightarrow \U_{i+1}(A)\rightarrow 0
\]
of graded $R$-modules. Applying the depth lemma~\cite[Proposition 1.2.9]{Bruns-Herzog} to the depth of $\U_{i+1}(A)$ gives us the implication \eqref{partb} $\implies$ \eqref{parta}.\\ To show the implication \eqref{parta} $\implies$ \eqref{partb}, we use induction on $\ell=d-i$. If $\ell=1$, then $R/I^{\langle d-1\rangle}=\U_d(A)$ which is Cohen--Macaulay, since $A$ is sequentially Cohen--Macaulay . Hence, depth of $R/I^{\langle d-1\rangle}$ is equal to $d$. Now, if we assume that $\dps\left(R/I^{\langle i+1\rangle}\right)\geq i+2$, then the depth lemma implies that $\dps\left(R/I^{\langle i\rangle}\right)\geq i+1$, since $\U_{i+1}(A)$ is Cohen--Macaulay by assumption.\\
Observe that the equivalence between parts \eqref{partb} and \eqref{partcc} follows from Lemma~\ref{dpsborel} and the fact that $\dps\left(R/I^\lir\right)=\dps\left(R/\mathrm{gin}(I^\lir)\right)$.\\
Finally, to see the equivalence between parts \eqref{partcc} and \eqref{partc}, we note that taking generic initial ideal of the both sides of $I\subseteq I^\lir$ together with Proposition~\ref{gindim} imply that 
\[
\mathrm{gin}\left(I\right)\subseteq \mathrm{gin}\left(I^\lir\right)\subseteq \mathrm{gin}\left(I\right)^\lir.
\]
Which implies that 
\[
\mathrm{gin}\left(I\right)^\lir\subseteq \mathrm{gin}\left(I^\lir\right)^\lir\subseteq \left(\mathrm{gin}\left(I\right)^\lir\right)^\lir=\mathrm{gin}\left(I\right)^\lir.
\]
\end{proof}

Now we are in position to state the main result of this section.

\begin{thm}\label{sqcm}
Let $I\subseteq R$ be a homogeneous ideal. Then $R/I$ is sequentially Cohen--Macaulay if and only if $$\BW\left(R/I;t,w\right)=\BW\left(R/\mathrm{gin}(I);t,w\right).$$
\end{thm}
\begin{proof}
This is an immediate consequence of Proposition~\ref{char}.
\end{proof}

We will end this section with an application to the algebraic shifting theory of simplicial complexes. Algebraic shifting, invented by Kalai, is an operator that associate to every simplicial complex a shifted complex, while preserving many interesting properties such as $f$-vector and topological Betti numbers. We refer the reader to the survey article by Kalai~\cite{Kalai} or the book by Herzog \& Hibi~\cite{Herzog-Hibi} for precise definetions as well as all undefined terminologies in the sequel. 

\begin{thm}\label{Symm}
Let $\k$ be a field of characteristic zero, $\Delta$ a simplicial complex and $\Delta^{\mathrm{s}}$ its symmetric algebraic shifting. Then $\Delta$ is sequentially Cohen--Macaulay (i.e., $\k[\Delta]$ is sequentially Cohen--Macaulay) if and only if $\Delta$ and $\Delta^{\mathrm{s}}$ have the same $h$-triangles. 
\end{thm}
\begin{proof}
We observe the simple fact that any shifted complex is sequentially Cohen--Macaulay. Hence, Theorem~\ref{sqcm} implies that $\k[\Delta^{\mathrm{s}}]$ has a stable Bj\"{o}rner--Wachs polynomial passing to the generic initial ideal. However, it can be easily deduced from~\cite[Proposition 11.2.9]{Herzog-Hibi} that the Stanley--Reisner ideals of $\Delta$ and $\Delta^{\mathrm{s}}$ have the same generic initial ideal. Therefore, we may conclude by using Theorem~\ref{sqcm} once more time. 
\end{proof}

\begin{rmk}

For a simplicial complex $\Delta$, the exterior non-face ideal $J_\Delta$ and the exterior face algebra $\k\{\Delta\}=E/J_\Delta$ can be defined analoguesly as the Stanley--Reisner ideal $I_\Delta$ and the face ring $\k[\Delta]$. The exterior algebraic shifting of $\Delta$ is then defined to be the simplicial complex $\Delta^{\mathrm{e}}$ such that $J_{\Delta^{\mathrm{e}}}=\mathrm{gin}(J_\Delta)$. See~\cite[Chapters 5 \& 11]{Herzog-Hibi} for exterior face rings and exterior algebraic shifting, or~\cite{Kalai} for a more combinatorial approach. Duval~\cite{Duval}, proved that a simplicial complex $\Delta$ is sequentially Cohen--Macaulay if and only if $\Delta$ and $\Delta^{\mathrm{e}}$ have the same $h$-triangles. It should be noted that the symmetric and exterior algebraic shifting do not necessarily coincide, even in the special case of sequentially Cohen--Macaulay complexes. As an example illustrating this one may consider the complete bipartite graph $K_{3,3}$, see e.g.~\cite[p. 128]{Kalai}. 

\end{rmk}


\section{Connections to Local Cohomology Modules}\label{Local}
In this section we show that for a sequentially Cohen--Macaulay algebra the Bj\"{o}rner--Wachs polynomial and the Hilbert series of local cohomology modules (supported on the maximal graded ideal $\mathfrak{m}$ of $R$) determine each other. We start by recalling some notions. The general references for the facts that we use here are the books by Bruns \& Herzog~\cite{Bruns-Herzog} and by Stanley~\cite{StanleyGreen}.

Let $M$ be a finitely generated $\mathbb{Z}$-graded $R$-module. The \Defn{injective hull} of $M$ is the smallest injective module containing $M$. The injective hull $\mathrm{E}_R(\k)$ of $\k\cong R/\mathfrak{m}$ as a graded $R$-module is isomorphic to $\k[x_1^{-1},\ldots,x_n^{-1}]$. The \Defn{Matlis dual} of $M$ is defined to be
\[
M^\vee:=\mathrm{Hom}_R\left(M,\mathrm{E}_R(\k)\right). 
\]

\begin{thm}\label{localCoh}
Let $A$ be a standard $\mathbb{N}$-graded sequentially Cohen--Macaulay algebra of Krull dimension $d$. Then one has \[(t-1)^i\mathrm{Hilb}\left(\mathrm{H}^i_\mathfrak{m}(A);t\right)=\mathrm{h}\left(\U_i(A);t\right)\]  for $1\leq i\leq d$. 
\end{thm}

\begin{proof}
First observe that, it follows from Grothendieck local duality~\cite[Theorem I. 12.3]{StanleyGreen} and~\cite[Proposition 1.3]{Herzog-Sbarra} that 
\[
\mathrm{H}^i_\mathfrak{m}(A)\cong\mathrm{Ext}^{n-i}_R\left(A,R(-n)\right)^\vee\cong\mathrm{Ext}^{n-i}_R\left(\U_i(A),R(-n)\right)^\vee\cong\mathrm{H}^i_\mathfrak{m}(\U_i(A))
\]
where $R(-n)$ is the canonical module of $R$ (the free module with one generator of degree $n$). So, we may assume that $A=\U_i(A)$ is Cohen--Macaulay of Krull dimension $i=d$. Now, if $y\in R$ is an $A$-regular degree-one form, then the sequence 
\[
0\rightarrow \mathrm{H}^{i-1}_{\mathfrak{m}}(A/yA)\rightarrow \mathrm{H}^{i}_{\mathfrak{m}}(A)(-1)\rightarrow \mathrm{H}^{i}_{\mathfrak{m}}(A)\rightarrow 0
\]
is exact (see, e.g., \cite[p. 176]{Bruns-Herzog}). So, one has 
\begin{equation}\label{eq8}
\mathrm{Hilb}(\mathrm{H}^{i-1}_{\mathfrak{m}}(A/yA);t)=(t-1)\cdot\mathrm{Hilb}(\mathrm{H}^{i}_{\mathfrak{m}}(A);t). 
\end{equation}
Thus, if $\Theta=(y_1,\ldots,y_i)$ is an $A$-regular sequence of $1$-forms, then, by repeating equation~\eqref{eq8} say, one obtains
\[
\mathrm{Hilb}(\mathrm{H}^{0}_{\mathfrak{m}}(A/\Theta A);t)=(t-1)^i\cdot\mathrm{Hilb}(\mathrm{H}^{i}_{\mathfrak{m}}(A);t).
\]
However, $A/\Theta A$ being zero-dimensional, we have $\mathrm{H}^{0}_{\mathfrak{m}}(A/\Theta A)=A/\Theta A$. On the other hand, $\mathrm{Hilb}(A/\Theta A;t)=\h(A;t)$ and we are done. 

\end{proof}
\begin{rmk}
It was shown by Herzog \& Sbarra~\cite{Herzog-Sbarra} that a finitely generated $\mathbb{Z}$-graded $R$-module is sequentially Cohen--Macaulay if and only if passing to the reverse lexicographic generic initial module preserves the Hilbert series of all local cohomology modules supported on $\mathfrak{m}$. Their result together with Theorem~\ref{localCoh} would imply one direction of our Theorem~\ref{sqcm}. 
\end{rmk}

 Theorem~\ref{localCoh}, in particular, says that in the sequentially Cohen--Macaulay case the Bj\"{o}rner--Wachs polynomial and the Hilbert series of local cohomology modules contain the same information. So, it wouldn't be surprising if we can detect the depth and the Castelnuovo--Mumford regularity from the information contained in the Bj\"{o}rner--Wachs polynomial. And, in fact, we have the following generalisation of the well-known result that the regularity of a Cohen--Macaulay module is equal to the degree of its $h$-polynomial.

\begin{cor}
Let $A$ be a standard $\mathbb{N}$-graded sequentially Cohen--Macaulay algebra. Then the extremal Betti numbers of $A$ can be read from its Bj\"{o}rner--Wachs polynomial. In particular,
\begin{enumerate}[\rm (a)]
\item the Castelnuovo--Mumford regularity of $A$ is equal to the largest integer $j$ such that $t^j$ appears with a non-zero coefficient in $\BW(A; t, w)$, and
\item the depth of $A$ is equal to the smallest integer $i$ such that $w^i$ appears with a non-zero coefficient in $\BW(A; t, w)$.
\end{enumerate}
\end{cor}\qed


\section{Comments on Combinatorial Alexander Duality}\label{Remarks}

We close this paper by two remarks on the Alexander dual of sequentially Cohen--Macaulay simplicial complexes.\\
 A special case of our Theorem~\ref{localCoh}, due to Herzog, Reiner \& Welker~\cite[Proposition 12]{HRW}, has been proved to be a very useful tool. For instance, it was recently used in~\cite{KAA} in order to derive a characterisation of the Betti table of componentwise linear ideals from a characterisation of the $h$-triangle of sequentially Cohen-Macaulay simplicial complexes. We shall take a closer look at this special case. Let $\Delta$ be a simplicial complex. In this situation, Hochster's Theorem~\cite[Theorem II. 4.1]{StanleyGreen} asserts that
\begin{equation*}
\mathrm{Hilb}\left(\mathrm{H}^i_\mathfrak{m}(\k[\Delta]);t\right)=\sum_{\substack{F\in\Delta\\ |F|=c}}\dim_\k \wt{H}_{i-c-1}(\mathrm{link}_\Delta F;\k)\frac{t^{-c}}{(1-t^{-1})^{c}}. 
\end{equation*}
Hence, if we denote by $\Delta^\ast$ the combinatorial Alexander dual of $\Delta$, then it follows from the dual version of Hochster's formula~\cite[Proposition 1]{Eagon-Reiner} that 
\begin{equation*}\label{localBetti}
\mathrm{Hilb}\left(\mathrm{H}^i_\mathfrak{m}(\k[\Delta]);t\right)=\sum_{c}\beta_{i-c+1,n-c}(\k[\Delta^\ast])\frac{t^{-c}}{(1-t^{-1})^{c}}. 
\end{equation*}
It follows then easily by combining Theorems~\ref{BWDelta} and~\ref{localCoh} and the formula above (together with a routine computation left to the reader) that 
\begin{cor}[{{\cite[Proposition 12]{HRW}}}]
Let $\Delta$ be a sequentially Cohen--Macaulay simplicial complex on $[n]$. Then, one has
\begin{equation}\label{eq:HRW}
\sum_c\beta_{i-c+1,n-c}(\k[\Delta^\ast])(t-1)^{i-c}=\sum_\ell h_{i,\ell}(\Delta)t^\ell,
\end{equation}
for all $0\leq i\leq d=\dim\k[\Delta]$ . 
\end{cor}\qed

\noindent It can be easily shown that exterior algebraic shifting and Alexander duality commute, see e.g.~\cite[Lemma 1.1]{Herzog-Terai}. On the other hand, it is a challenging open problem to show that symmetric algebraic shifting and Alexander duality also commute. In practice, however, it is sometimes enough to know that some invariants of $\k[(\Delta^\mathrm{s})^\ast]$ and $\k[(\Delta^\ast)^\mathrm{s}]$ coincide. For instance, the proof of Theorem~4.1 in~\cite{Herzog-Sbarra} relies (implicitly) on the property that: 

\begin{prop}
For a sequentially Cohen--Macaulay complex $\Delta$ and for all $i$ and $j$ one has \(\beta_{i,j}(\k[(\Delta^\mathrm{s})^\ast])=\beta_{i,j}(\k[(\Delta^\ast)^\mathrm{s}]).\)
\end{prop}\qed

\noindent While we are not aware of any reference showing this property, it can be derived easily from our Theorem~\ref{Symm} and equation~\eqref{eq:HRW}. Also, note that the proof of equation~\eqref{eq:HRW} in~\cite{HRW} is done without any reference to the Hilbert series of local cohomology modules. So, the above-mentioned property can be used to give a direct simple proof that a complex is sequentially Cohen--Macaulay if and only if the Hilbert series of its local cohomology modules are stable under symmetric algebraic shifting, as in~\cite[Theorem 4.1]{Herzog-Sbarra}.




\paragraph{Acknowledgement.} I would like to thank Mats Boij, Aldo Conca, Alex Engstr\"{o}m, Ralf Fr\"{o}berg, Enrico Sbarra and  Volkmar Welker for some many helpful conversations. In particular, Theorem~\ref{localCoh} is inspired from a conversation that I had with Aldo Conca.

{\small

\def\cprime{$'$}
\providecommand{\bysame}{\leavevmode\hbox to3em{\hrulefill}\thinspace}
\providecommand{\MR}{\relax\ifhmode\unskip\space\fi MR }
\providecommand{\MRhref}[2]{%
  \href{http://www.ams.org/mathscinet-getitem?mr=#1}{#2}
}
\providecommand{\href}[2]{#2}

}

\end{document}